\documentclass[a4paper,10pt]{article}
\usepackage[sumlimits,intlimits]{amsmath}
\usepackage{amsfonts}
\usepackage{amssymb}
\usepackage{stmaryrd}
\usepackage{array}
\usepackage{graphicx} 

\usepackage{nomencl}

\newcommand{\mykill}[1]{ }

\usepackage[thmmarks,amsmath,hyperref,thref]{ntheorem}
\newcommand{\thmspace}{\bigskip}

\usepackage{hyperref} 

\theoremprework{\thmspace}
\newtheorem{Th}{Theorem}[section]
\theoremprework{\thmspace}
\newtheorem{Lemma}[Th]{Lemma}
\theoremprework{\thmspace}
\newtheorem{Proposition}[Th]{Proposition}
\theoremprework{\thmspace}
\newtheorem{Corollary}[Th]{Corollary}

\theorembodyfont{\upshape}
\theoremprework{\thmspace}

\theoremprework{\thmspace}
\newtheorem{definitionremark}[Th]{Definition/Remark}
\theoremprework{\thmspace}
\newtheorem{remark}[Th]{Remark}
\theoremprework{\thmspace}

\newtheorem{ex}[Th]{Example}
\theoremstyle{nonumberplain}
\qedsymbol{\ensuremath{_\square}}
\theoremsymbol{\ensuremath{_\square}}
\newtheorem{proof}{Proof}

\newcommand{\R}{\mathbb{R}}

\newcommand{\C}{\mathbb{C}}
\newcommand{\N}{\mathbb{N}}

\renewcommand{\d}{\partial}
\DeclareMathOperator{\W}{W}

\DeclareMathOperator{\Gr}{Gr}
\DeclareMathOperator{\Der}{Der}
\DeclareMathOperator{\id}{id}
\DeclareMathOperator{\ann}{Ann}
\DeclareMathOperator{\lm}{lm}
\DeclareMathOperator{\MPUM}{MPUM}
\DeclareMathOperator{\VMPUM}{VMPUM}
\DeclareMathOperator{\Shift}{Shift}
\DeclareMathOperator{\Sol}{Sol}
\DeclareMathOperator{\HP}{HP}
\DeclareMathOperator{\Syz}{Syz}
\DeclareMathOperator{\Bild}{Im}
\DeclareMathOperator{\const}{const}

\date{}
\title{Exact linear modeling using Ore algebras}
\author{Kristina Schindelar, Viktor Levandovskyy\footnote{Corresponding author. E-Mail: {\tt Viktor.Levandovskyy@math.rwth-aachen.de}}, Eva Zerz\\
Lehrstuhl D f\"ur Mathematik, RWTH Aachen University \\ 52062 Aachen, Germany}

\begin{document}

\maketitle

\begin{abstract}
Linear exact modeling is a problem coming from system identification:
Given a set of observed trajectories, the goal is find a model (usually,
a system of partial differential and/or difference equations)
that explains the data as precisely as possible.
The case of operators with constant coefficients is well studied and known
in the systems theoretic literature, whereas the operators
with varying coefficients were addressed only recently.
This question can be tackled either using Gr\"obner bases for
modules over Ore algebras or by following the ideas from
differential algebra and computing in commutative rings. 
In this paper, we present algorithmic methods to compute ``most powerful
unfalsified models'' ($\MPUM$) and their counterparts with
variable coefficients ($\VMPUM$) for polynomial and
polynomial-exponential signals.  
We also study the structural properties of the
resulting models, discuss computer algebraic techniques behind algorithms
and provide several examples.
\end{abstract}

Key words: Ore algebra, noncommutative Gr\"obner basis, annihilator,
syzygies, linear exact modeling.

MSC 2000 classification: 13P10, 93B25, 68W30, 93A30.

\tableofcontents 

\section{Introduction}
Linear exact modeling is a problem of system identification that leads
to interesting algebraic questions. We start with some motivation
from the systems theoretic point of view:
The problem of linear exact modeling was formulated for one-dimensional 
behaviors in \cite{AW}, see also \cite{KP1,KP2}. Starting with an observed set of 
polynomial-exponential signals, the aim is to find a linear 
differentiation-invariant model for these.
Evidently, the whole signal set is a behavior that is not falsified by 
observation. But such a model has no significance. Making the behavior 
larger than necessary, the accuracy of the explanation decreases. 
So besides the condition that the desired model should be unfalsified, 
we are searching for the most powerful one. This means that the model 
does not admit more solutions than necessary. A model satisfying all 
conditions is abbreviatory called continuous $\MPUM$ 
(most powerful unfalsified model).

In \cite{Eva1}, the modeling was extended to multidimensional behaviors 
\cite{CQR,PQ},  
and in \cite{Eva2} to the discrete framework, that is, instead of the 
requirement that the model should contain all derivatives of the signals, 
it is required that all shifts of the signals are contained.

In other words, the problem is to find a homogeneous system of partial differential 
equations with constant coefficients that is as restrictive as possible with the property 
of possessing the observed signals as solutions.

In \cite{VMPUM} a different approach was introduced.
There the goal is to find all partial differential equations with 
polynomial coefficients that are solved by the signals. 
Thus the new aspect of this approach is the choice of a different model class.
Indeed the properties of the resulting model depend strongly 
on the model class. For instance, by the 
transition from the $\MPUM$ to the $\VMPUM$, the time-invariance vanishes.
In this paper, we continue this approach. But since the continuous case is 
not the only interesting one, we will consider 
a more general problem comprising 
both the continuous and the discrete situation.
Later some special model classes will be discussed in more detail.

Let us particularize our goal. 
Let $K$ be a field and $O$ be an operator algebra over $K$. Further 
let $\mathcal{A}_O$ be a function space over $K$ possessing an $O$-module structure.
A \textbf{model} or a so-called \textbf{behavior} $\mathcal{B}$ is the solution
set 
of a homogeneous linear system, given by finitely many equations. These equations are defined in terms of
the operator algebra $O$. Thus $\mathcal{B}$ is characterized by
$$ \mathcal{B}=\Sol(O^{1\times r}R)= \lbrace \omega \in \mathcal{A}_O^m \, | \, R \bullet \omega= 0 \rbrace, \;\;\; 
\mbox{ where } \; R\in O^{r \times m}$$
and $\bullet$ denotes the natural extension of the module action $o \bullet \omega$ of 
$o \in O$ on $\omega \in \mathcal{A}_O$ to the matrix $R \in O^{r \times m}$ and the vector
$\omega \in \mathcal{A}_O^m$.
In most cases of interest, we have $K\subseteq O$ and $o k = k o$ for all $k \in K$, $o \in O$.
Then $\mathcal{B}$ is a $K$-vector space, and thus the introduced model class is linear.
Within such a model class we want to perform modeling now. Suppose we observe a set of signals 
$\Omega \subseteq \mathcal{A}_O^m$. The aim is to find a model $\mathcal{B}_\Omega$ in the model class such that
\begin{enumerate}
 \item $\mathcal{B}_\Omega$ is unfalsified by $\Omega$, i.e. \ $\Omega \subseteq \mathcal{B}_\Omega$.
 \item $\mathcal{B}_\Omega$ is most powerful, i.e.\ for every behavior $\mathcal{B}$ with
$\Omega \subseteq \mathcal{B}$, it follows that $\mathcal{B}_\Omega \subseteq \mathcal{B}$.
\end{enumerate}
If $\mathcal{B}_\Omega$ is invariant under the action of $O$, that is,
if we have for all $o \in O$ 
$$ \omega \in \mathcal{B} \;\; \Rightarrow  \;\; o \bullet \omega \in \mathcal{B},$$ 
it is called \textbf{most powerful unfalsified model}, 
short $\MPUM$ of $\Omega$. Else, if $\mathcal{B}_\Omega$ varies under $O$ it is called 
\textbf{variant most powerful unfalsified model}, short $\VMPUM$ of $\Omega$. 
We denote the $\VMPUM$ of $\Omega$ by $\mathcal{B}_\Omega^V$.

The following example shows how the choice of the model class affects the model.

\begin{ex}\label{MotivatingExample}
Consider the signal set consisting of a single signal 
$$\Omega= \lbrace \omega \rbrace, \;\;\; \mbox{where } \; \omega(t)=t \;  \mbox{ for all  } t \in \R .$$ 
\begin{enumerate}
\item Let $O=\C[\partial]$ and $\mathcal{A}_O=\mathcal{C}^\infty(\R,\C)$, where 
$\partial \bullet f := \frac{df}{dt}$. Using the commutative structure of the operator ring, the underlying system is invariant under 
differentiation:
$$ R \bullet w = 0 \;\;\; \Rightarrow \;\;\; R ( \partial \bullet w ) = 
( R \partial ) \bullet w =  ( \partial R ) \bullet w
=   \partial ( R \bullet w) =0.$$
Since we are searching for a differentiation-invariant model, we obtain that besides $\omega$, also 
its derivative, the constant function $1$, belongs to $\mathcal{B}_{\Omega} $. Using that the model is $\C$-linear, we get that
$$\mathcal{B}_{\Omega} = \{ w  \ | \ \exists a,b \in \C: \forall t\in \R : w(t)=at+b \}.$$
An element $w \in  \mathcal{C}^\infty(\R,\C)$ is contained in $\mathcal{B}_\Omega$ if and only if
$$\partial^2 \bullet w=0,$$
i.e. the $\MPUM$ is specified by a single ordinary differential equation with constant coefficients.
\item Now let $O=\C[t]\langle \partial \rangle$, where 
$\partial \bullet f := \frac{df}{dt}$ and $\mathcal{A}_O$ is defined as above. We want to describe 
$\omega$ as a solution of homogeneous ordinary differential equations with polynomial coefficients. The 
equations $$\partial^2 \bullet w=0 \;\; \mbox{ and } \;\; t \partial \bullet  w-w=0 $$
are satisfied by $\omega$. We will see later that these two generate a kernel representation of the $\VMPUM$ of $\Omega$.
The corresponding solution space equals $$\mathcal{B}_\Omega^V=\{ w \ | \ \exists a \in \C : \forall t\in{\R} : w(t)=at\}.$$
Notice that this example demonstrates the variance under $\partial$, since we have
$ \partial \bullet \omega \notin \mathcal{B}_\Omega^V$.
Another property that should be pointed out is that the $\VMPUM$ yields a more
precise description of $\Omega$ than $\MPUM$.
\end{enumerate}
\end{ex}

\section{Ore algebras}

The example above deals with continuous signals. But in applications, there are also 
discrete phenomena or combinations of discrete and continuous signals that are of great interest too.
Many of the relevant operator algebras have the structure of an Ore algebra, as studied e.g. in
\cite{CQR, CQR2, CS}. 
We give a definition that is motivated by \cite{CS}. Moreover, this simplifies more
general setup of \cite{Kr}.

Hence first consider skew polynomial rings, a generalization of polynomial rings to 
the noncommutative framework. 

\begin{definitionremark} \cite{MR} \
\label{defOre}
\begin{enumerate} 
\item[(1)] Let $A$ be a ring and $\sigma : A \rightarrow A$ be a ring endomorphism.
\begin{enumerate} 
\item The map $\delta : A \rightarrow A$ is called a $\sigma$-\textbf{derivation} if it is 
$K$-linear and satisfies the skew Leibniz rule
\begin{align}\label{deltaOperation}
\delta(a b) = \sigma(a) \delta(b) + \delta(a) b \;\;\;\; \mbox{for all } a, b \in A.
\end{align}
\item For a $\sigma$-derivation $\delta$, the ring $A\left[ \partial;\sigma, \delta \right] $ 
which
consists of all polynomials in $\partial$ with coefficients in $A$ with the usual addition 
and a product defined by the commutation rule
$$\partial a = \sigma(a) \partial + \delta(a) \;\;\;\; \mbox{for all } a \in A,$$
is called a \textbf{skew polynomial ring} or an \textbf{Ore extension} of $A$ with $\sigma$ and $\delta$.
\end{enumerate}
\item[] If $A$ is a domain and $\sigma$ is injective, the 
skew polynomial ring $A\left[ \partial;\sigma, \delta \right]$ is a domain by degree arguments. 
Then the definition can be iterated to the so-called Ore algebras \cite{CS}.
\item[(2)] Let $A=K[t_1, \dots, t_n]$. An iterated skew polynomial ring 
$$O=K[t_1, \dots, t_n] [\partial_1;\sigma_1, \delta_1]  \cdots [\partial_s;\sigma_s, \delta_s]$$ 
is called a (polynomial) \textbf{Ore algebra} if the $\sigma_i$'s and $\delta_j$'s commute for $ 1 \leq i , j \leq s $, 
the $\partial_i$'s commute with $\partial_j$'s and further for all $1\leq i \leq s$ the map $\sigma_i:O \rightarrow O$ is 
an injective $K$-algebra endomorphism and $\delta_i:O \rightarrow O $ is a $\sigma_i$-derivation satisfying
$ \sigma_i( \partial_j)= \partial_j \;\;\;\; \mbox{ and }\;\;\;\; \delta_i(\partial_j)=0. $
\end{enumerate}
\end{definitionremark}

Using multi-index notation, every element of an Ore algebra can be expressed into the normal form
$$
\sum_{\alpha \in \N_0^s} p_\alpha \partial^\alpha  = 
\sum_{\alpha \in \N_0^s} p_\alpha {\partial_1}^{\alpha_1} \cdot \ldots \cdot {\partial_s}^{\alpha_s}  
\;\;\;\; \mbox{ where } \;\; p_\alpha \in A.
$$

For our issues the most interesting examples of Ore algebras are the following ones.

\begin{ex}\label{interestingAlgebras} 
Let $n=1$, thus $A=K[t]$. The algebras can be iterated to $n \in \N$.
\begin{enumerate}
\item The first \textbf{Weyl algebra} is defined by
$\W_1:= A[ \partial; \id_{ \W_1}, \frac{\partial}{\partial t}]$
with the commutation rule $ \partial  t = t \partial + 1 .$
\item The first \textbf{difference algebra} is defined by
$ \mathcal{S}_1 :=A \left[ \Delta; \sigma, \delta \right]$,
where $ (\sigma p)(t) =  p(t + 1)$ and $\delta(p)=\sigma(p)-p$ for all 
$p \in \mathcal{S}_1$. The commutation rule is
$ \Delta t = t \Delta + \Delta + 1.$
\item The following Ore algebra is a combination of the first and second one. Define
$\mathcal{SW}_1:=A \left[ \Delta; \sigma_1, \delta_1 \right]\left[ \partial; \sigma_2, \delta_2 \right]$,
where $\sigma_2:=id_{ \mathcal{SW}_1 }$, $\delta_2:=\frac{\partial}{\partial t}$ and $ (\sigma_1 p)(t) = p(t + 1)$, 
$\delta_1(p)=\sigma_1(p)-p$ for all $p \in \mathcal{SW}_1$. Then $ \partial t = t \partial + 1$,
$\Delta t = t \Delta + \Delta +1$ and $\partial \Delta = \Delta \partial$.
\item Suppose $q$ to be a parameter. The first \textbf{continuous $q$-difference algebra} is defined by
$ \mathcal{Q}:= A [\partial; \sigma, \delta], $
where $\sigma( p )=p(q t)$ and $\delta(p)= p(qt)-p(t)$. We obtain the commutation rule
$ \partial t = qt \partial + (q-1)t$.
\end{enumerate}
\end{ex}


\begin{Lemma}
\label{OreIso}
Let $A$ be a ring, and $A[\partial; \sigma,\delta]$ be an Ore extension of $A$. 
For any $\alpha \in A$ there exists an Ore extension $A[\Delta_{\alpha}; \sigma,\delta']$
with $\delta'(a) = \sigma(a) \alpha - \alpha a + \delta(a)$, such that 
$A[\partial; \sigma,\delta] \cong A[\Delta_{\alpha}; \sigma,\delta'] $ as rings.
\end{Lemma}
\begin{proof}
For all $a \in A$, the equality $\partial a = \sigma(a) \partial + \delta(a)$ holds. For $\alpha \in A$ define
$\Delta_{\alpha} := \partial - \alpha$. Then it obeys the relation 
$\Delta_{\alpha} a = \sigma(a) \Delta_{\alpha} + \sigma(a) \alpha - \alpha a + \delta(a) 
= \sigma(a)\Delta_{\alpha} + \delta'(a)$.
The map $\delta'$ is linear and it is a $\sigma$-derivation since
\[
\delta'(ab) = \sigma(a) \sigma(b) \alpha - \sigma(a) \alpha b + \sigma(a) \delta(b) -
\sigma(a) \alpha b - \alpha ab + \delta(a)b = 
\] 
\[
\sigma(a) \sigma(b) \alpha - \alpha ab + \sigma(a) \delta(b) - \delta(a)b = 
\sigma(ab) \alpha - \alpha ab + \delta(ab).
\]
Define the ring homomorphism
$\varphi_{\alpha}: A[\partial; \sigma,\delta] \to A[\Delta_{\alpha}; \sigma,\delta']$,
$\varphi_{\alpha}(a)=a$ for all $a \in A$, $\varphi_{\alpha}(\partial)=\Delta_{\alpha}=\partial - \alpha$.
Then $\varphi_{\alpha}$ is an isomorphism.
\end{proof}

Let 
$O:=
A[\partial_1; \sigma_1,\delta_1]\cdots[\partial_m; \sigma_m,\delta_m]$ be an Ore algebra.
 With the action
$$
\partial_i \bullet p := \delta_i(p)  \;\;\; \mbox{ and } \;\;\; a \bullet p :=a \cdot p \;\;\; \mbox{ for all } p\in A \mbox{ and } a \in A
$$
the $K$-algebra $A$ becomes an $O$-module. For this, we have to show that 
\begin{enumerate}
 \item $ (o_1 \cdot o_2) \bullet p = o_1 \bullet ( o_2 \bullet p ) \;\;\;\; $ for all $o_1, o_2 \in O$ and $p \in A$ 
 \item $ (o_1 + o_2) \bullet p = o_1 \bullet p +  o_2 \bullet p \;\;\;\; $ for all $o_1, o_2 \in O$ and $p \in A$ 
 \item $ o \bullet (p+q) = o \bullet p + o \bullet q \;\;\;\; $ for all $o \in O$ and $p, q \in A.$ 
\end{enumerate}
To show 1. it suffices to consider $o_1=a \partial_i$ and $o_2= b \partial_j$ with $a, b \in A$. Then
\begin{eqnarray*}
 (o_1 \cdot o_2) \bullet p & = & (\, a( \, \sigma_i(b) \partial_i + \delta_i(b)\,) \partial_j \,) \bullet p 
 =  (\, a\sigma_i(b) \partial_i \partial_j + a\delta_i(b) \partial_j \,) \bullet p \\
& = & a\sigma_i(b) \delta_i(p) \delta_j(p) + a\delta_i(b) \delta_j(p)
 =  a \, \delta_i (b\delta_j(p))
 =  a \partial_i \bullet (\, b \partial_j \bullet p \,)\\
& = &  o_1 \bullet (\, o_2 \bullet p \,).
\end{eqnarray*}
The equality in 2. and 3. holds by similar arguments. \\

Using this action, we can define the kernel of a linear operator $f$ from the Ore algebra $O$ over a ring $A$ to be $\ker_{A} f := \{ a\in A \mid f \bullet a = 0\}$, which is a $K$-vector space.

\begin{Lemma}
\label{CorDelta}
Let $K$ be a field, $A$ be a $K$-algebra, $\partial$ be a $K$-linear operator, acting on $A$ and $B = A[\partial; \sigma, \delta]$ be the corresponding operator algebra (that is, for all $a \in A$ we have $ \partial a = \sigma(a) \partial + \delta(a)$). Then the following holds:

(i) $\ker_A \partial = A \Leftrightarrow \delta = 0 \Leftrightarrow B = A[\partial; \sigma, 0]$.

(ii) If $\ker_A \partial = A$, then we have for $\Delta := \partial - 1$:
$A[\partial; \sigma, 0]$ is isomorphic as $K$-algebra to operator algebra $A[\Delta; \sigma, \delta']$ with $\delta' := \sigma - 1$. Moreover, 
$\ker_A \Delta = \{a\in A\mid \sigma(a) = a\} = \const_{\sigma} A \subseteq A$ with the
equality if and only if $A$ is invariant under $\sigma$, what is the case if $\sigma = 1_A$.
\end{Lemma}



\begin{remark}
Using Lemmata \ref{OreIso} and \ref{CorDelta}, we pass to the new setting of operators, 
which action $\bullet$ is nontrivial on $A$. We call such an operator nontrivial and
from now on, we work with such operators only.
\end{remark}

\begin{ex}
Consider the two most important operator algebras, built from operators having zero kernels.
The first \textbf{forward shift algebra} is defined by
$K[t] \left[ s; \sigma, 0 \right]$ with $(\sigma f)(t) =  f(t + 1)$ for all $f \in K[t]$.
The commutation rule is $s t = t s + s$. There is a natural operator associated to $s$, namely the difference operator $\Delta = s-1$, already defined in \ref{interestingAlgebras}, obeying the relation $\Delta t = t \Delta + \Delta + 1$. Applying Lemma \ref{CorDelta}, we see by degree argument, that $\ker \Delta = K$ and the two algebras are isomorphic both as Ore extensions and $K$-algebras. \\
Let $q$ be transcendental over $K$. Then the first \textbf{$q$-commutative algebra} (or Manin's quantum plane) is defined as $K_q[x,y]:=K(q)[x][\partial;\sigma,0]$ with $(\sigma f)(x)=f(qx)$ for $f\in K[t]$. Again, there is a natural $q$-difference operator $\Delta_q := \partial - 1$ and the corresponding operator algebra has been already described in 
\ref{interestingAlgebras} as the first  continuous $q$-difference algebra. Its commutation rule reads as $\partial t = qt \partial + (q-1)t$.
\end{ex}

For $o_1, \dots, o_k \in O^{n}$, we denote by ${}_O \langle o_1, \dots o_k \rangle$ the left submodule of $O^n$, generated by $o_1, \dots, o_k$.
\begin{Th}\label{IsoOrePoly}
Let $O$ be an Ore $A$-algebra, built from operators $\partial_1, \dots, \partial_s$ which have non-zero kernels. Then there is an isomorphism of left $O$-modules 
$$ O/ {}_O \langle \partial_1, \dots, \partial_s \rangle \cong A. $$
\end{Th}
\begin{proof}
There is a left $O$-module homomorphism 
$$
\varphi : O \rightarrow A, \;\;\; a=\sum_{\alpha \in \N_0^s} a_{\alpha} \partial^\alpha \mapsto a \bullet 1 
$$
since $\varphi(b\cdot a) = (b\cdot a) \bullet 1= b\bullet \varphi(a)$. 
Due to Def. \ref{defOre} (\ref{deltaOperation}) we have $\delta(1)=0$ and thus $a \bullet 1=a_0$. 
The kernel of $\varphi$ is given by the left ideal 
${}_{O} \langle \partial_1, \dots, \partial_s \rangle $. Further,~$\varphi$ 
is clearly surjective. So the claim follows from the homomorphism theorem. 
\end{proof}

Following Theorem \ref{IsoOrePoly}, every polynomial $p \in A$ can be 
viewed as an element of the left $O$-module 
$O / {}_O \langle \partial_1, \dots, \partial_s \rangle$ by identifying 
$p$ with $p + {}_{O} \langle \partial_1, \dots, \partial_s \rangle=:[p]$. Then the 
action of $\partial_i$ is exactly the $\sigma_i$-derivation $\delta_i$, since 
\begin{eqnarray*}
\partial_i [p] = [ \partial_i p] = [ \sigma_i(p) \partial_i + \delta_i(p)] = 
[\delta_i(p)] = [\partial_i \bullet p].
\end{eqnarray*}

\begin{remark}\label{ann}
Let $p \in A$ and $o \in O$. Then there is the following equivalence
\begin{eqnarray*} 
o \bullet p =0 \;\;\; \mbox{ if and only if } \;\;\; 
o \cdot p \in {}_{O} \langle \partial_1, \dots, \partial_s \rangle.
\end{eqnarray*}
\end{remark}
\begin{proof}
By Theorem \ref{IsoOrePoly}, we have an $O$-module isomorphism 
$A \!\cong\! O / {}_O \langle \partial_1, \dots, \partial_s \rangle$ given by 
$$
A \stackrel{\cong}{\longrightarrow} O / 
{}_O \langle \partial_1, \dots, \partial_s \rangle, \;\;\;\;\;  p \mapsto [p].
$$
Since the $O$-module structure is respected, 
$o \bullet p$ maps to $[o \cdot p]$ and hence
the claim follows. 
\end{proof}

Remark \ref{ann} gives the possibility to describe and to compute the 
annihilator of an element $p \in A$. 
Consider the map
\begin{equation} \label{kappa_p}
\kappa_p : O \rightarrow O /  {}_O\langle\partial_1, \dots, \partial_s \rangle, 
\;\;\;\;\; o \mapsto o \cdot [p],
\end{equation}
which is clearly a left $O$-module homomorphism with the kernel
$$\ker(\kappa_p) = \ann_O(p):=\{ o \in O \ | \ o \bullet p=0  \}, $$ 
which is a left ideal in $O$. See Corollary \ref{corBasics} for its algorithmic computation.
This construction lifts to the case of vectors. Suppose 
$p= \left[ p_1, \dots, p_m \right]^{T}  \in A^m$. 
An element of $o \in O^{1 \times m}$ naturally acts on $p$ by 
$$o \bullet p := \sum_{i=1}^m o_i \bullet p_i .$$ 
A subset $B \subseteq A^m$ is called \textbf{invariant} under 
$G \subseteq O^{1 \times m}$ if and only if $o \bullet p = 0$ for all 
$o \in G$ and $p \in B$.
The set of elements under which $p$ is invariant has an $O$-module 
structure and equals to the kernel of 
$$
\kappa_p : O^{1 \times m} \rightarrow O /  {}_O \langle \partial_1, \dots, \partial_s \rangle, 
\;\;\;\;\; o=[o_1, \dots, o_m] \mapsto \sum_{i=1}^m o_i \cdot [p_i].
$$
Moreover the following isomorphism holds
$$O^{1\times m} / \ker(\kappa_p) \cong 
{}_{O}\langle p_1, \dots, p_m \rangle / {}_{O}\langle p_1, \dots, p_m \rangle \cap 
{}_{O}\langle \partial_1, \dots, \partial_s \rangle .$$

The image of $\kappa_p$ equals $( {}_{O} \langle p_1, \dots, p_m \rangle + {}_{O}\langle 
\partial_1, \dots, \partial_s \rangle )  /
{}_{O}\langle \partial_1, \dots, \partial_s \rangle.$ This is isomorphic to 
$ {}_{O}\langle p_1, \dots, p_m \rangle / 
{}_{O}\langle p_1, \dots, p_m \rangle \cap {}_{O}\langle \partial_1, \dots, \partial_s \rangle$. 
So the claim follows, since $\kappa_p$ is a homomorphism.

\begin{remark}\label{finitely}
If $O$ is Noetherian (see \cite{MR}), then the left submodule 
$\ker(\kappa_p) \subseteq O^{1 \times m}$ is finitely generated. \\

For a polynomial $m$-tuple $p\in A^m$, we consider 
\[
\ann_O (p) = \{ o\in O \mid o \bullet p = 0 \} = 
\{ o\in O \mid o \bullet p_i = 0 \ \forall \ i \} = 
\bigcap \ann_O (p_i),
\]
which is a left ideal in $O$. As we see immediately,
$\ann_O (p)^{1\times m}$ is a (usually strict) submodule
of $\ker(\kappa_p)$ and hence, the latter typically has 
more interesting structure, see Example \ref{exKerP}. It is
always possible to recover $\ann_O (p)$ from $\ker(\kappa_p)$.
In our opinion, using $\ker(\kappa_p)$ is more natural in the
context of vectors of signals.
\end{remark}

\section{Algorithmic computations}

For the concrete calculations used in this article, we need 
algorithms for the following computational tasks over (polynomial) 
Ore algebras:

\begin{enumerate}
\item syzygy module of a tuple of vectors
\item elimination of module components from a submodule of a free module
\item annihilator ideal of an element in a finitely presented module
\item kernel of a homomorphism of modules
\item intersection of a finite number of submodules of a free module.
\end{enumerate}

Let $O$ be a Noetherian Ore algebra. 
Moreover, let $M$ be a finitely presented left $O$-module, that is, there exists a matrix $P\in O^{m \times n}$ such that there is the following exact sequence of left $O$-modules:
\[
O^{1\times m} \stackrel{P}{\rightarrow} O^{1 \times n} \rightarrow M \rightarrow 0.
\]

Recall that for a tuple $F= (f_1,\ldots,f_s), f_i \subset O^{1\times n}$, 
the set
$\mathrm{LeftSyz}(F) := \{ [a_1,\ldots,a_s] \in O^{1\times s} \mid  \sum_i a_i f_i = 0 \}$ 
carries the structure of a left $O$-module and is called the \textbf{left syzygy module} of $F$. Since $O$ is Noetherian, $\mathrm{LeftSyz}(F)$ is finitely generated. 
Computation of syzygies over Noetherian Ore algebras can be accomplished with
several algorithms and requires Gr\"obner basis techniques; see \cite{Kr} for Ore algebras and \cite{GPS05} for the commutative case. \\

Let $\{e_i \}$ be the canonical basis of the free module 
$O^{1\times \ell} = \overset{\ell}{\underset{i=1}{\bigoplus}} O e_i$.


\begin{Proposition}\label{EliOfCom}
\label{propBasics}
\begin{enumerate}
\item \textit{``Elimination of module components''}. \\

Let $S\subset O^{1\times \ell}$ be a submodule. Moreover, let $<_O$ be a monomial ordering on $O$ and $<_m=(c,<_O)$ be a \textit{position over term} monomial module ordering on the free module $O^{1\times \ell}$, defined as follows.  The components are ordered in a descending way $e_1 > \dots > e_{\ell}$ and for any monomials $o_1, o_2 \in O$
\[
o_1 e_i <_m  o_2 e_j \; \Leftrightarrow \; j < i \; \text{ or } \; (j = i 
\text{ and } o_1 <_O o_2).
\]
Let $G$ be a Gr\"obner basis of $S$ with respect to $<_m$. Then $\forall \; 1\leq k < \ell \;\;$ $G \cap \overset{\ell}{\underset{i=k}{\oplus}} O e_i$ is a Gr\"obner basis of $S \cap\overset{\ell}{\underset{i=k}{\oplus}} O e_i$. 

\item \textit{``Kernel of a homomorphism of modules''}. \\

Consider an $O$-module homomorphism $O^{1 \times s} \stackrel{\psi}{\rightarrow} O^{1 \times n}/ O^{1\times m}P $, $e_i \mapsto [\Psi_i]$, where $\Psi_i \in O^{1 \times n}$. Let $P_i$ be the $i$-th row of the matrix $P$. Then
\[
\ker \psi = \mathrm{LeftSyz}( \ (\Psi_1,\ldots,\Psi_s,P_1,\ldots,P_m) \ ) \cap \bigoplus_{i=1}^s O e_i.
\]

\end{enumerate}
\end{Proposition}

\begin{proof}
\begin{enumerate}
\item Define $W= \overset{\ell}{\underset{i=k}{\bigoplus}} O e_i$. Since $G$ is a Gr\"obner basis of $S$, for any $s \in S$ there exists $g \in G$ such that $\lm(g)$ divides $\lm(s)$. If $s\in S \cap W$, then $\lm(g) \in W$ and hence, by definition of $<_m$, we have $g \in W$ and $g \in G \cap W$. So, $G \cap W$ is a Gr\"obner basis of $S\cap W$.
\item We have \[
[b_1,\ldots,b_s] \in \ker \psi \ \Leftrightarrow \ \exists a_k\in O \ : \ 
\sum_{i=1}^s b_i \Psi_i + \sum_{k=1}^m a_k P_k =0 \ \Leftrightarrow 
\]
\[
[b_1,\ldots,b_s] \in \mathrm{LeftSyz}(\ (\Psi_1,\ldots,\Psi_s,P_1,\ldots,P_m)\ ) \cap \bigoplus_{i=1}^s O e_i.
\]

\end{enumerate}
\end{proof}

\begin{Corollary}
\label{corBasics}
\begin{enumerate}
\item \textit{``Annihilator of an element in a module''}. \\

Let $M = O^{1 \times n}/O^{1\times m}P$ and let $P_1,\ldots,P_m$ be the 
rows of $P$.
Moreover, let $v \in O^{1\times n}$. Then the left ideal 
$\ann_M^{O} (v) := \{ a\in O \mid a[v] = 0 \in M\}\subseteq O$ 
can be computed as
\[
\ann_M^{O} (v) 
= \ker (O \stackrel{\cdot [v]}{\rightarrow} M)
= \mathrm{LeftSyz}(\ (v,P_1,\ldots,P_m)\ ) \cap O e_1.
\]


\item \textit{"Intersection of finitely many submodules''}. \\

Let $N_1, \ldots, N_m \subset O^{1\times r}$ be submodules. Then
\[
\bigcap_{i=1}^m N_i = \ker \bigl(
O^{1\times r} \rightarrow (O^{1\times r} / N_1) \oplus \cdots \oplus (O^{1\times r} / N_m),
\;\; e_i \mapsto ([e_i],\ldots,[e_i]) \bigr).
\]
\end{enumerate}
\end{Corollary}

\begin{remark}
For an $O$-module homomorphism 
$O^{1 \times s}/O^{1\times r}Q \!\stackrel{\psi'}{\rightarrow} \!
O^{1 \times n}/O^{1\times m}P$, its kernel is the image of 
$\ker \psi$ (as in Theorem~\ref{propBasics}) under the natural projection 
$O^{1 \times s} \rightarrow O^{1 \times s}/ O^{1\times r}Q$. 
A left Gr\"obner basis can be obtained by reducing a
left Gr\"obner basis of $\ker \psi + O^{1\times r}Q$ with a left 
Gr\"obner basis of $O^{1\times r}Q$, see \cite{LV05}.\\

Note that in practical computations, elimination of module components 
is usually not complicated. This stands in distinct contrast with the 
elimination of algebra variables, which is often very hard to achieve. 
The algorithms used in this article involve only the elimination of 
module components and thus are feasible in practice.\\

The algorithms we have discussed are implemented in computer algebra systems like e.g. \textsc{Singular::Plural} \cite{Plural} or {\sc Maple} \cite{CQR, CS}  with the package \textsc{OreModules}. More background on these algorithms can be found in e.g. \cite{Kr}, \cite{LV05}. \\

In particular, a set of generators of $\ker(\kappa_p)$ from the
previous section can be calculated explicitly.

By Proposition \ref{EliOfCom}(2) $\ker(\kappa_p)$ is obtained via the kernel of a module homomorphism, that is, by one Gr\"obner basis computation with respect to module monomial ordering eliminating components. The monomial part of this ordering can be chosen arbitrarily to be e.~g.~a fast one.
\end{remark}

\section{Application to linear exact modeling}

We will now use the results from above to define an unfalsified and most 
powerful model over an Ore algebra. 

\textbf{Assumptions and notations:}
Suppose $O$ to be a Noetherian Ore algebra with the additional property that 
$\partial_i$ acts nontrivially on $A$ for all $1 \leq i \leq s $.

Recall that $\mathcal{A}_O $ denotes a function space over $K$ possessing 
an $O$-module structure. Suppose further that $A \subseteq \mathcal{A}_O$.

\begin{remark}\label{noetherian} $[$ \cite{MR}, Theorem 1.2.9.$]$ 
Since $A$ is Noetherian, $O$ is Noetherian if
$\sigma_i$ is an automorphism for all $1\leq i\leq s$ on $A$.
\end{remark}
Thus all Ore algebras considered 
in Example \ref{interestingAlgebras} are Noetherian.

Suppose that the $\sigma$-derivation acts nontrivially on $A$. Then the 
corresponding most powerful unfalsified model varies under $O$.

Starting with a single signal $p \in A^m$, we want to find the $\VMPUM$ of $p$,
that is a behavior, invariant under some finitely generated submodule of $O^{1 \times m}$.

\begin{Th}\label{vmpum} Let $p\in A^m$ be given. Consider the map
$\kappa_p$ from (\ref{kappa_p}).
Let $\ker(\kappa_p)={}_{O} \langle k_1, \dots, k_r \rangle$ 
and let $R \in O^{r \times m}$ be a matrix whose
$i$-th row equals $k_i$. 
Then the $\VMPUM$ of $ \lbrace p \rbrace$ is given by
$$ \mathcal{B}_{\left\lbrace p \right\rbrace }^V=
\left\lbrace g\in \mathcal{A}_O^m \; | \; R \bullet g=0 \right\rbrace. $$
\end{Th}
\begin{proof}
By the definition of $R$ and Remark \ref{ann},  it is clear that 
$\left\lbrace p \right\rbrace \subseteq \mathcal{B}_{\left\lbrace p \right\rbrace }^V $.

It remains to show that $\mathcal{B}_{\left\lbrace p \right\rbrace }$ is most powerful. 
Suppose there exists another behavior $\mathcal{B}'$ unfalsified by $p$. The behavior $\mathcal{B}'$ 
possesses a kernel representation 
$R' \in O^{r' \times m}$.
By the definition of $R$, there exists a matrix $X \in O^{r' \times r}$ such that $R'=XR$. 
But since $(X \cdot R) \bullet p = X \bullet (R \bullet p)$, it follows that 
$\mathcal{B}_{\left\lbrace p \right\rbrace }^V \subseteq \mathcal{B}'$.
\end{proof}

\begin{ex}\label{VMPUMex}
 Let us consider a more interesting example than Example \ref{MotivatingExample}
with respect to our favorite algebras from Example \ref{interestingAlgebras}.  Let $\Omega=\{\omega\}$ 
consists of the cuspidal cubic $$\omega(t_1,t_2)=t_1^3-t_2^2.$$ 
Let us denote by $\mathcal{A}_O=\C[[t_1,t_2]]$ the ring of formal power series and consider the 
$\VMPUM$ $ \mathcal{B}_{\left\lbrace \omega \right\rbrace }^V= \lbrace f \in
\mathcal{A}_O \; | \; R_{\VMPUM} \bullet f= 0 \rbrace $ of $\Omega$ with respect to several
operator algebras $O$.
\begin{enumerate}
 \item Suppose $O$ to be the second \textbf{Weyl algebra} (see Example \ref{interestingAlgebras}). Then
by using {\sc Singular} we obtain: 
$$ R_{\VMPUM} =  
\left[ \begin{array}{c} 
\partial_2^3 \\
\partial_1 \partial_2 \\
\partial_1^3 + 3 \partial_2^2 \\
t_2 \partial_2^2-\partial_2 \\
t_2 \partial_1^2 + 3 t_1 \partial_2\\
2 t_1 \partial_1+3 t_2 \partial_2-6
\end{array} \right]. $$ 
Now let us determine $\mathcal{B}_{\left\lbrace \omega \right\rbrace }^V$ to see how 
precise the description given by the $\VMPUM$ is. Let $f \in \mathcal{A}_O$.
\begin{enumerate}
\item $\partial_2^3 \bullet f = 0 \;\; \Rightarrow \;\; f=c_0 + c_1 t_2 + c_2 t_2^2, \;$ where
      $c_i \in \C[[t_1]]$.
      
\item $\partial_1 \partial_2 \bullet f = 0 \;\; \Rightarrow \;\; \partial_1\bullet c_1 + 2 t_2 \partial_1 \bullet c_2=0 
       \;\; \Rightarrow \;\; \partial_1\bullet c_1 =0  \; \wedge \; \partial_1\bullet c_2 =0$
      $  \Rightarrow \;\; c_1, c_2 \in \C $.
\item $ (\partial_1^3 + 3 \partial_2^2) \bullet f=0 \;\; \Rightarrow \;\; 
        \partial_1^3 \bullet c_0 +6 c_2=0 \;\; \Rightarrow \;\; 
        c_0 = -c_2 t_1^3 + d_2t_1^2 + d_1 t_1 +d_0, \;$ where $ d_i \in \C$.
\item $(t_2 \partial_2^2-\partial_2) \bullet f=0 \;\; \Rightarrow \;\; c_1=0 $.
\item $(3 t_1 \partial_2+t_2 \partial_1^2) \bullet f=0 \;\; \Rightarrow \;\;
       d_2=0$. 
\item $(2 t_1 \partial_1+3 t_2 \partial_2-6) \bullet f = 0 \;\; \Rightarrow \;\; 
       -4d_1t_1-6d_0 = 0 \;\; \Rightarrow \;\;d_1=0=d_0$.
\end{enumerate}
Hence, we obtain that $f=c(t_1^3 - t_2^2)$, thus 
$$ \mathcal{B}_{\left\lbrace \omega \right\rbrace }^V= \{ c (t_1^3 - t_2^2) \ | \ c\in \C \}.$$
With respect to the requirement of being most powerful and linear, 
the $\VMPUM$ is as significant as possible.
We observe that the $\VMPUM$ of a single non-zero 
signal has $\C$-dimension one. 
Actually, this holds in general, as will be shown in Theorem
\ref{CDimension}.
\item  Suppose $O$ to be the second \textbf{difference algebra} see Example \ref{interestingAlgebras}. Then
by using {\sc Singular} we obtain: 
$$R_{\VMPUM} = 
\left[ \begin{array}{c} \Delta_2^3\\
\Delta_1 \Delta_2\\
\Delta_1^3+3 \Delta_2^2\\
2 t_2 \Delta_2^2+\Delta_2^2-2 \Delta_2\\
2 t_2 \Delta_1^2+\Delta_1^2+6 t_1 \Delta_2+6 \Delta_2\\
8 \Delta_1^2 +21 \Delta_2^2 + 24 t_1 \Delta_1+36 t_2 \Delta_2-24 \Delta_1-18 \Delta_2-72
\end{array} \right]. $$ 
Similar arguments as above lead us to
$$\mathcal{B}_{\left\lbrace \omega \right\rbrace }^V= \{ c (t_1^3 - t_2^2) \ | \ c\in \C \}.$$
\item Suppose $O$ to be the second \textbf{$\mathcal{SW}$ algebra} see Example \ref{interestingAlgebras}. Then
by using {\sc Singular} we obtain:
$$
R_{\VMPUM} = 
\left[ \begin{array}{c}
\Delta_2^3\\
\Delta_1\Delta_2\\
\Delta_1^3+3\Delta_2^2\\
2\partial_2+\Delta_2^2-2\Delta_2\\
2\partial_1+\Delta_1^2-2\Delta_1+2\Delta_2^2\\
2t_2\Delta_2^2+\Delta_2^2-2\Delta_2\\
2t_2\Delta_1^2+\Delta_1^2+6t_1\Delta_2+6\Delta_2\\
8\Delta_1^2+21\Delta_2^2+24t_1\Delta_1+36t_2\Delta_2-24\Delta_1-18\Delta_2-72
\end{array}\right]. $$
Note that generators in the output depend on the monomial ordering of the operators. 
In this example $\Delta_{1,2}$ were chosen to be greater 
that $\d_{1,2}$. Taking a reverse ordering
produces different (but equivalent) answer.

Comparing this matrix with the matrix above, we see that the rows of the matrix belonging to
the difference case appear also here. We conclude that 
$$\mathcal{B}_{\left\lbrace \omega \right\rbrace }^V= \{ c (t_1^3 - t_2^2) \ | \ c\in \C \}.$$
Thus, taking $\mathcal{SW}$ as operator algebra, we have got more equations 
than with the difference algebra. However, we have obtained very interesting mixed differential-difference equations, which show the interplay of two different operator settings.

 \item  The second \textbf{$q$-difference algebra} see Example \ref{interestingAlgebras}:
\begin{small} $$
R_{\VMPUM} = \left[ \begin{array}{c}
\partial_2^2+(-q^2+1) \partial_2\\
(-q-1) \partial_1+(-q^2-q-1) \partial_2+(q^4+q^3-q-1)\\
t_1^3 \partial_2-t_2^2 \partial_2+(q^2-1) t_2^2
\end{array}\right]. $$ \end{small}
\begin{enumerate}
\item The first equation yields:
\begin{align*}
&\;\; \sum_{i,j}c_{i,j}(q^j-1)^2t_1^it_2^j + (-q^2+1) \sum_{i,j}c_{i,j}(q^j-1)t_1^it_2^j = 0 \\
\Leftrightarrow & \;\;(q^j-1)^2+(-q^2+1)(q^j-1)=0 \\
\Leftrightarrow & \;\; j=0 \vee j=2.
\end{align*}
\item Now consider the second equation. 
\begin{enumerate}
\item Suppose $j=2$, then
\begin{align*}
& \;\;(-q-1)\sum_{i,j}c_{ij}(q^i-1)t_1^it_2^2  +(-q^2-q-1) \sum_{i,j}c_{ij}(q^2-1)t_1^it_2^2\\
& \;+(q^4+q^3-q-1)\sum_{i,j}c_{ij}t_1^it_2^2=0\\
\Leftrightarrow & \;\; (-q-1)(q^i-1)+(-q^2-q-1)(q^2-1)+(q^4+q^3-q-1)=0 \\
\Leftrightarrow & \;\; i=0.
\end{align*}
\item Suppose $j=0$, then
\begin{align*}
& \;\;(-q-1)\sum_{i,j}c_{i0}(q^i-1)t_1^i +(q^4+q^3-q-1)\sum_{i}c_{i0}t_1^i=0 \\
\Leftrightarrow & \;\; i=3.
\end{align*}
Thus \quad $f = c_{30} t_1^3 + c_{02}t_2^2$.
\end{enumerate}
\item Applying the last equation, we get 
\begin{align*}
\;\;t_1^3 c_{02}(q^2-1)t_2^2 -t_2^2 c_{02}(q^1-1)t_2^2 +(q^2-1) t_2^2 ( c_{30} t_1^3 + c_{02}t_2^2 ) =0\\
\Leftrightarrow \;\; t_1^3t_2^2(q^2-1)( c_{30} + c_{02} ) = 0 \;\; \Leftrightarrow \;\; c_{30} =-c_{02}.
\end{align*}
\end{enumerate}
Thus we obtain once more $\mathcal{B}_{\left\lbrace \omega \right\rbrace }^V= \{ c (t_1^3 -t_2^2) \ | \ c\in \C \}$.
\end{enumerate}
\end{ex}

\begin{remark}
As we have seen in the previous example, 
the number of equations giving the $\VMPUM$
depends strongly on the underlying Ore algebra. 
In all cases, with Gr\"obner bases we get more equations than it might be actually necessary. 
However, it is possible to compute a smaller generating set, which is usually not a Gr\"obner basis.
Namely, one computes a left syzygy module of a given system and almost directly deduces 
a smaller generating set from it. As an example, we show that only 3 of 6 equations from the
first example of \ref{VMPUMex} generate the whole ideal,
namely $\d_1 \d_2, \d_1^3+3 \d_2^2, 2 t_1 \d_1+ 3 t_2 \d_2-6$. Analogous smaller generating sets
can be obtained for other examples.
\end{remark}

Theorem \ref{vmpum} can be generalized to a set of several signals directly. 
A kernel representation of the $\VMPUM$ of 
$\Omega=\{\omega_1, \dots, \omega_N \}$ is determined by stacking a set of 
generators of $$\bigcap_{i=1}^N \ker( \kappa_{\omega_i} )$$
row-wise into a matrix $R$. 

\begin{Th}
Using the notation from above, the $\VMPUM$ of $\Omega$ equals 
$$\mathcal{B}_{\Omega}^V= \left\lbrace g \in \mathcal{A}_O^m \; | \; R \bullet g=0 \right\rbrace.$$
\end{Th}
\begin{proof}
By the definition of $R$, it is clear 
that $\Omega \subseteq \mathcal{B}_\Omega^V$. Also the property 
of being most powerful follows by the same arguments 
as used in the proof of Theorem~\ref{vmpum}.
\end{proof}

\begin{ex}
Suppose $O$ to be the first Weyl algebra and $\mathcal{A}_O=\mathcal{C}^\infty(\R,\C)$.  
Consider the signal set $\Omega = \{ t, v_0 t-v_1 t^2\}$, 
where $v_0, v_1 \in \C \setminus \{0\}$. 
The second trajectory will appear in Example $\ref{wurf}$ again. Since
\begin{align*}
& \ker(\kappa_t) \cap \ker(\kappa_{v_0t-v_1t^2})\\  & = {}_{\W_1} \langle t \partial-1,\ \partial^2 \rangle  
\cap {}_{\W_1} \langle -v_0^2 \partial^2+(4v_1^2 t- 2v_0v_1)\partial-8v_1^2 ,\ \partial^3 \rangle \\
& = {}_{\W_1} \langle t^2 \partial^2 - 2 t \partial + 2 , \partial^3 \rangle, 
\end{align*}
the $\VMPUM$ of $\Omega$ is given by
$ \mathcal{B}_\Omega^V=\{ c_1 t + c_2 t^2 \ | \ c_1, c_2 \in \C \}$.
The intersection of submodules of a free module 
over a Noetherian Ore algebra can be computed as in 
Corollary \ref{corBasics}, 
for instance with the system \textsc{Singular::Plural} \cite{Plural}. 
\end{ex}

\section{VMPUM by using the polynomial Weyl algebra}

In this section, we suppose $O$ to be the $n$-th Weyl algebra 
\begin{center}
$O = \W_n:= \C[t_1, \dots, t_n][ \partial_1; \id_{ \W_n}, \frac{\partial}{\partial t_1}] 
\cdots [ \partial_n; \id_{ \W_n}, \frac{\partial}{\partial t_n}].$
\end{center}
Thus for $p \in\C[t_1, \dots, t_n]$, we obtain $ \partial_i \bullet p := \frac{\partial p}{\partial t_i}$.
Further suppose $\mathcal{A}_O $ to be $\mathcal{C}^\infty(\R^n,\C)$, the space of 
smooth functions. Identifying a polynomial with the corresponding polynomial function, we
obtain $A \subseteq \mathcal{A}_O$.

In this context, the $\VMPUM$ was already introduced in \cite{VMPUM}. 
Here, we will recall some results and additionally 
point out a new interesting property. 

\subsection{$\C$-dimension}
A known result is that the $\VMPUM$ is a finite-dimensional
vector space over~$\C$, since it is contained in the 
corresponding $\MPUM$ \cite{VMPUM}.
In some cases, we can determine the dimension more precisely. 
We claim that the $\VMPUM$ of
a single non-zero signal has $\C$-dimension one.

Suppose $p \in A^m$. Every polynomial $p_i$ can be written as 
$\sum_{k=1}^{h_i}c_{ik}t^{\beta_{ik}}$, 
where $c_{ik} \in \C$ for all $i,k$. Let 
$\mathcal{E}_i:=\left\lbrace \beta_{i1}, \dots, 
\beta_{ih_i} \right\rbrace \subset \N_0^n$ denote the set of all 
exponent multi-indices occurring in $p_i$ and let
\begin{eqnarray}\label{expo}
d_{ij}:=\max_{1\leq k \leq h_i} \left\lbrace (\beta_{ik})_j \; | \; \beta_{ik} \in 
\mathcal{E}_i \right\rbrace 
\end{eqnarray}
be the highest degree in $t_j$ of $p_i$. Recall that by $e_i$ we denote the $i$-th canonical 
generator of the free module $A^m$. The set 
$$
\mathcal{E}_{p_i}=\{ \alpha \in {\N}_0^n \, | \, \alpha_j \leq d_{ij}+1 \; \mbox{ for } \, 1\leq j \leq n \}
$$
is finite, that is, $\mathcal{E}_{p_i} = \lbrace \alpha_{i1}, \dots, \alpha_{il_i} \rbrace$.
Define for $p\in A^m$
$$ \Der_p=\lbrace p_1, \frac{\partial^{|\alpha_{11}|}p_1}{\partial^{\alpha_{11}}}, 
\dots,\frac{\partial^{|\alpha_{1l_1}|}p_1}{\partial^{\alpha_{1l_1}}}, \, \dots \, , p_m, 
\frac{\partial^{|\alpha_{m1}|}p_m}{\partial^{\alpha_{m1}}}, \, \dots \,, 
\frac{\partial^{|\alpha_{ml_m}|}p_m}{\partial^{\alpha_{ml_m}}} \rbrace.$$
Let $\Syz(\Der_p)$ denote the module of polynomial syzygies.
Define for the matrix
$M=[ 1 e_1, \; \partial^{\alpha_{11}} e_1, \; \dots, \;
\partial^{\alpha_{1l_1}}e_1, \; \dots, \; 1 e_m, \; \partial^{\alpha_{m1}}e_m, \;
\dots, \; \partial^{\alpha_{ml_m}} e_m]^T$
the $A$-module homomorphism 
$$\Phi_p:\Syz(\Der_p) \rightarrow \ker(\kappa_p), \;\;\; (q_1, \dots, q_l) \mapsto  (q_1, \dots, q_l) \cdot M, 
$$
which is clearly injective. 



\begin{Lemma}
\label{KomCom}
${}_{\W_n}\langle \Bild(\Phi_p) \rangle = \ker(\kappa_p)$.
\end{Lemma}

\begin{proof}
Evidently ${}_{\W_n} \langle \Bild(\Phi_p) \rangle \subseteq \ker(\kappa_p)$. Now suppose that 
$a \in \ker(\kappa_p)$. Since every element in $\W_n$ can be written in 
normal form, 
we obtain
\begin{align*}
 a \bullet p & = \sum_k a_k \bullet p_k = \sum_k ( \sum_j c_{kj} t^{\beta_{kj}} \partial^{\gamma_{kj}} ) \bullet p_k \\
& = \sum_k ( \sum_j c_{kj} t^{\beta_{kj}})( \partial^{\gamma_{kj}} \bullet p_k ).
\end{align*}
Let us split the element $a$ in $a_z$ and $a_{nz}$ such that $a=a_z+a_{nz}$ and $(a_z)_k$ consists of the parts of $a_k$ where $\partial^{\gamma_{kj}} \bullet p_k$ is zero. 

By the choice of $d_{ij}$, the set
$\{ \partial_j^{(d_{ij} +1)} \, | \, 1\leq j \leq n, \; 1 \leq i \leq m \}$
generates the set of $\partial^\gamma$
with the property that there exists $1 \leq i \leq m$ such that 
$ \partial^\gamma \bullet p_i=0.$
Then $a_z$ is contained 
${}_{W_n}\langle \partial^{(d_{ij} +1)e_j} \, | \, 1\leq j \leq n, \; 1 \leq i \leq m \rangle$. 
But by the choice of $\Der_p$, the element $a_z$ is in the image of $\Phi_p$. 
Suppose $\partial^{\gamma_{kj}} \bullet p_k \neq 0$, then $\gamma_{kj}$ is equal or smaller than 
$(d_{k1}, \dots, d_{kn})$ in each component and again 
by the choice of $\Der_p$, the element $a_{nz}$ is contained in the image of $\Phi_p$. 
Thus it follows that $a \in \Bild(\Phi_p)$.
\end{proof}

\begin{Th}\label{CDimension}
 The $\VMPUM$ of $p \neq 0$ is a one-dimensional vector space over~$\C$.
\end{Th}
\begin{proof}
We use the notation of the Lemma \ref{KomCom}, which reduces to commutative calculations. It is easy to see that the equivalence  
\begin{align}\label{SolDModCom}
 s \bullet (f_1, \partial^{\alpha_{11}} \bullet f_1, \dots, f_m, \dots, \partial^{\alpha_{mh_m}} \bullet f_m)^T = 0 
\;\;\;\; \Leftrightarrow \;\;\;\; \Phi_p(s) \bullet f =0 
\end{align}
holds for every $ s \in \Syz(\Der_p) $. Now let us discuss the left hand side. More precisely, let us consider the solution space $\Sol(\Syz(\Der_p))$ in $\mathcal{A}_O$ belonging to $\Syz(\Der_p)$. Since $\Der_p$ contains 
all non-zero derivatives of $p_i$ for all $i$, there exists a non-zero constant ${\C} \ni k \in \Der_p$. 
We can suppose $k=1$ and without loss of generality let 
$\partial^{\alpha_{mh_m}} \bullet p_m=1$. Then 
$$(-1, 0, \dots, 0, p_1), \; \dots, \;
(0, \dots, 0, -1, \frac{\partial^{|\alpha_{mh_m}|}p_m}{\partial^{\alpha_{mh_m}}}) \in \Syz(\Der_p)$$
and thus 
\begin{align}\label{SolCom}
 \Sol(\Syz(\Der_p)) = \lbrace c \cdot (p_1, \frac{\partial^{|\alpha_{11}|}p_1}{\partial^{\alpha_{11}}}, \dots, \frac{\partial^{|\alpha_{mh_m}|}p_m}{\partial^{\alpha_{mh_m}}}) \; | \; c \in {\C} \rbrace.
\end{align}
Now suppose $f \in \VMPUM_p$. From Lemma \ref{KomCom} together with (\ref{SolDModCom}) and (\ref{SolCom}) we deduce the claim.
\end{proof}

\begin{remark}
 In the case of a single non-zero signal, the $\VMPUM$ gives the most precise description one can get for a linear system.
\end{remark}

\begin{ex}\label{wurf}
Consider the trajectory
$\omega(t)=v_0 t - v_1 t^2$, where $v_0,v_1\in{\C}\setminus\{0\}$.
Then the $\MPUM$ of $\omega$ is given by
\begin{align*}
 \mathcal{B}_{\{\omega\}} &= \{ \alpha (v_0t-v_1t^2)+\beta(v_0-2v_1t)+
\gamma (-2v_1) \ | \ \alpha, \beta, \gamma \in \R  \} \\
  & = \{ at^2 + bt +c \ | \ a, b, c \in \R \}\\
  & =  \{ w \in \mathcal{A}_O \ | \ \partial^3 \bullet w = 0 \}. 
\end{align*}
Thus there are three free parameters to choose. 
The $\VMPUM$ of $\omega$ is given by
\begin{align*}
 \mathcal{B}^V_{\{ \omega\}} & 
=\{ w \in \mathcal{A}_O \ | \  \left[ \begin{array}{c} -v_0^2 \partial^2+(4v_1^2 t- 2v_0v_1)\partial-8v_1^2  \\ \partial^3 \end{array} \right]  \bullet w = 0 \} \\
& = \{ c \, ( v_0 t - v_1 t^2 ) \ | \ c \in \R \},
\end{align*}
that is, two degrees of freedom vanish when we consider the time-variant model.
\end{ex}

\subsection{Structural properties}

Let us discuss some structural properties of the $\W_n$-module $\ker(\kappa_p)$.
Since every element of $\W_n$ can be transformed into normal form, 
the degree of an element $a\in\W_n$ can be introduced as
$$ \deg(a):= \max \lbrace \; \sum_{i=1}^n \alpha_i + \beta_i  \;\; | \;\; a = \sum_{\alpha, \beta \in \N_0^n} a_{\alpha, \beta} 
t^{\alpha} \partial^{\beta}, \; a_{\alpha, \beta} \in \C \rbrace. $$ 
Then $\mathcal{F}^i(\W_n):=\left\lbrace a \in \W_n \; | \; \deg(a)\leq i \right\rbrace $ induces a filtration 
on $\W_n$. The 
corresponding associated graded ring $\Gr(\W_n)$ is isomorphic to 
$\C[t_1, \dots, t_n, \!\partial_1, \dots, \partial_n]$ as a graded $\C$-algebra.
For every finitely generated $\W_n$-module $M$, 
we define the Hilbert polynomial 
$\HP_M^{\W_n}:=\HP_{\Gr(M)}^{\Gr(\W_n)}$. 
The dimension of $M$ is defined as $\dim_{\W_n}(M):=\deg(\HP_M)+1.$ 
Furthermore $M$ is called \textbf{holonomic} if it has dimension $n$. A holonomic module is of 
minimal dimension, since the dimension of $\W_n$-modules is bounded below by $n$ and bounded 
above by $2n$. Holonomic $\W_n$-modules are additionally cyclic and torsion modules. For details see \cite{C}.

As usual, we write $A=\C[t_1,\ldots,t_n]\subset \W_n$.

\begin{Th}\label{WeylCAlgebra}
There is an isomorphism of $\W_n$-modules
$ \W_n^{1\times m} / \ker(\kappa_p) \cong A.$
In particular, $\W_n^{1\times m} / \ker(\kappa_p)$ is simple holonomic $\W_n$-module.
\end{Th}
\begin{proof}
 Since $\kappa_p$ is a homomorphism of $\W_n$-modules, we get
$$ \W_n^{1\times m} / \ker(\kappa_p) \cong \Bild(\kappa_p) \subseteq \W_n / {}_{\W_n} \langle \partial_1, 
\dots, \partial_n \rangle \cong A. $$ 
Thus $ \W_n^{1\times m} / \ker(\kappa_p) $ is isomorphic to a submodule of $A$. 
Due to the fact that $A$ is a simple holonomic $\W_n$-module and $ \W_n^{1\times m} / \ker(\kappa_p) \neq 0$ 
the claim follows.
\end{proof}

\begin{Corollary}\label{cyclic}
 Since $\W_n^{1\times m} / \ker(\kappa_p)$ is holonomic, there exists a left 
ideal $L_p$, depending on $p$, such that $\W_n^{1\times m} / \ker(\kappa_p) $ 
is isomorphic to the cyclic left $\W_n$-module $\W_n / L_p $. 
\end{Corollary}

An algorithm, using Gr\"obner bases, to compute a generator of a holonomic module
is given in \cite{L04}. On the other hand, since $\W_n^{1\times m} / \ker(\kappa_p)$ is simple holonomic module, any non-zero element can be taken as a generator for a cyclic presentation.

\begin{ex}
Suppose $\omega=[ c_1, \; c_2, \; c_3 ]^T$ for $c_1, c_2, c_3 \in \R\setminus\{0\}$. Then 
$$ 
\ker(\kappa_\omega)=_{\W_1}\langle[0,c_3,-c_2], [c_3,0,-c_1], [0,0,\partial]\rangle.
$$
Since
$$
\left[ \begin{array}{ccc} 
     0 & c_3 & -c_2 \\ c_3 & 0 & -c_1 \\ 0 & 0 & \partial
       \end{array}
\right]
\cdot
\underbrace{\left[ \begin{array}{ccc} 
0 & 1/c_3 & c_1/c_3 \\ 1/c_3 & 0 & c_2/c_3 \\ 0 & 0 & 1
       \end{array}
\right]}_{:=C}
=
\left[ \begin{array}{ccc} 
     1 & 0 & 0 \\ 0 & 1 & 0 \\ 0 & 0 & \partial
       \end{array}
\right], 
$$
we obtain
$
\W_1^3/\ker(\kappa_{\omega}) \cong \W_1^3/ \ker(\kappa_\omega) C \cong \W_1/ _{\W_1}\langle \partial \rangle \cong \C[t]
$.
\end{ex}



\subsection{VMPUM of polynomial-exponential signals}

In this section, we extend the signal space that should be modeled. The goal is to compute the $\VMPUM$ of 
\begin{eqnarray}\label{expvektor}
p=[p_1 \exp_{\lambda^1}, \; \dots, \; p_m \exp_{\lambda^m}]^T ,
\end{eqnarray}
where for all $1 \leq i \leq m$, we have $p_i \in A$, $\lambda^i \in \C^n$ and 
$$ \exp_{\lambda}:=\exp(\lambda_1 t_1 + \dots + \lambda_n t_n) \;\;\; \; \mbox{ for } \; \lambda \in \C^n.$$
By the action 
$\partial_j \bullet \exp_\lambda = \lambda_j \exp_\lambda$ for all $1 \leq j \leq n$
the space of polynomial-exponential functions becomes a $\W_n$-module. 

Consider the scalar setting first, that is, $m=1$. Define for $\lambda \in \C^n$ the 
$\W_n$-homomorphism
\begin{eqnarray*}
\sigma_{\lambda}: \W_n \rightarrow \W_n, \;\;\; \partial_i \mapsto (\partial_i - \lambda_i), \;\; t_i \mapsto t_i .
\end{eqnarray*}
It is easy to see that $\sigma_{\lambda}$ is a $\W_n$-automorphism. 
We claim that for $a \in \W_n$ and $f \in A$
\begin{eqnarray}\label{einD} 
a \bullet p =0 \;\;\; \mbox{ if and only if } \;\;\; \sigma_{\lambda}(a) \bullet (p \exp_{\lambda})=0.
\end{eqnarray}
For the proof suppose $a= \sum_i c_i t^{\alpha_i} \partial^{\beta_i}$. 

Using the identity 
$(\partial_i - \lambda_i) \bullet (p \exp_{\lambda}) = (\partial_i \bullet p)\exp_{\lambda} $, 
the claim follows by
\begin{eqnarray*}
\sigma_{\lambda}(a) \bullet (p \exp_{\lambda}) 
&=& \sum_i c_i t^{\alpha_i}(\, (\partial_1-\lambda_1)^{{\beta_i}_1} \cdots (\partial_n-\lambda_n)^{{\beta_i}_n}\,) \bullet (p \exp_{\lambda})\\
& = &\sum_i c_i t^{\alpha_i}(\, (\partial_1^{{\beta_i}_1} \cdots \partial_n^{{\beta_i}_n}\,) \bullet p) \, \exp_{\lambda}
= \quad (a\bullet p)\exp_{\lambda} .
\end{eqnarray*}

Extending the dimension, there are two special cases requiring attention. 
First suppose $\lambda^1, \dots, \lambda^m$ to be equal, that is, $p=[p_1, \dots, p_m]^{T} \exp_{\lambda}$, 
where $\lambda:=\lambda^1$. Then claim (\ref{einD}) can be generalized directly and it follows that
\begin{eqnarray}\label{mD}
\sum_{i=1}^m a_i \bullet (p_i \exp_{\lambda}) = 0  \; \mbox{ if and only if } \; [a_1, 
\dots, a_m] \in \sigma_{\lambda} ( \ker(\kappa_{p}) ) .
\end{eqnarray}

Assume now that $\lambda^1, \dots, \lambda^m$ are pairwise different. Then 
\begin{eqnarray}\label{pwverschieden}
\sum_{j=1}^m a_j \bullet (p_j \exp_{\lambda^j}) =0 \;\;\; \mbox{if and only if}\;\;\;[a_1, \dots, a_m] \in 
\bigoplus_{j=1}^m \sigma_{\lambda^j}(\,\ker(\kappa_{p_j})  \,).
\end{eqnarray}
Since $\exp_{\lambda^1}, \dots,\exp_{\lambda^m} $ are algebraically independent over $A$, the claim follows from
\begin{eqnarray*}
& &\sum_{j=1}^m a_j \bullet (p_j \exp_{\lambda^j}) = 0 \\  
&\Leftrightarrow& \sum_{j=1}^m \left( \sum_{i=1}^{h_j} c_{ji} t^{\alpha_{ji}} (\partial_1 + \lambda^j_1)^{(\beta_{ji})_1} \dots
 (\partial_n + \lambda^j_n)^{(\beta_{ji})_n} \bullet p_j \right) \exp_{\lambda^j} = 0 \\
&\Leftrightarrow&\sigma_{\lambda^j}^{-1}(a_j) \, \in \, \ker(\kappa_{p_j}) \;\; \mbox{ for all } \, 1 \leq j \leq m.
\end{eqnarray*}

Recapitulating we get:

\begin{Th}
Let $f$ be of the form (\ref{expvektor}). Further let 
\[ K_i:= \left\lbrace j \; | \; \lambda^j= \lambda^i \right\rbrace = \left\lbrace  k_{i1}, \dots, k_{i l_i}\right\rbrace \] 
and let $l$ be 
chosen minimal such that we have a disjoint union $K_1 \dot \cup \dots \dot \cup K_l = \left\lbrace k_{11}, \dots, k_{1h_1}, \dots, k_{l1}, \dots, k_{lh_l} \right\rbrace =\left\lbrace 1, \dots, m \right\rbrace $. Further define the vector $h_i:=[f_{k_{i1}}, \dots, f_{k_{il_i}}]^{T}$ and $H_i:=\sigma_{\lambda^i}( \ker(\kappa_{h_i}) )$. Let  $e_{k_{ij}}$ denote the $k_{ij}$-th canonical generator of $\W_n^{1 \times m}$ for $1\leq i \leq l$ and $1 \leq l \leq h_i$. Defining for $1 \leq i \leq l$,
\begin{eqnarray*}
\phi_i: H_i \rightarrow \W_n , \;\;\; [a_1, \dots, a_{h_i}] \mapsto \sum_{j=1}^{h_i} a_j e_{k_{ij}}, 
\end{eqnarray*}
the $\VMPUM$ of $f$ is given by $ \bigoplus_{i=1}^l \phi_i(H_i)$.
\end{Th}
\begin{proof}
After choosing a suitable projection, the claim follows by (\ref{einD}) and (\ref{mD}).
\end{proof}

\section{VMPUM via the polynomial difference algebra}

Suppose that $|K|=\infty$. 
Recall the definition of the $n$-th difference algebra:
$$ 
\mathcal{S}_n:=K[t_1, \dots, t_n][\Delta_1; \sigma_1, \delta_1] \cdots [\Delta_n; \sigma_n, \delta_n].
$$
For $p\in K[t_1, \dots, t_n]$, we have
$$
\Delta_i \bullet p = \delta_i(p)=\sigma_i(p)-p=p(t+e_i)-p(t).
$$

Further suppose that $\mathcal{A}_O=K^{\N_0^n}$. 
Identifying a polynomial with the corresponding polynomial function, we
obtain $A \subseteq \mathcal{A}_O$.

Similar to the continuous case, the kernel of $\kappa_p$ can be 
computed in a completely 
commutative framework. For this we choose a special representation 
of the polynomials 
that is adapted to the action of $\Delta$, see \cite{Eva2}.
For $t \in \N_0^n$ and $\nu=(\nu_1, \dots, \nu_n)$, we consider 
the binomial functions
$$ p_{\nu}: \N_0^n \rightarrow K, \;\;\; t \mapsto  \binom{ t_1 }{ \nu_1 }
\cdots \binom { t_n }{ \nu_n  },$$
where $\binom{t_i }{ 0 } =1$ for all $i$. 
Then 
$ \nu! \, p_{\nu}=t_1 \cdots (t_1-\nu_1+1)  \cdots  t_n \cdots (t_n-\nu_n+1)$
and moreover, each element $p \in A^m$ can be written as

\begin{eqnarray}
p = \sum_{\nu \in \N_0^n, \nu \leq_{cw} \varrho} c_{\nu} p_{\nu} \label{darstellungvonp} 
\end{eqnarray}
for $\varrho \in \N_0^n$, some suitable 
coefficient vectors $c_{\nu} \in K^m$ and 
$\leq_{cw}$ denoting the component-wise order on ${\N}_0^n$, 
that is, $\nu_i \leq \varrho_i$ 
for all $1 \leq i \leq n$. 
Let us describe how to find this representation. We restrict to the 
scalar and one-dimensional case, 
where $m=n=1$. The general case can be treated similarly. 
For $p \in A=K[t]$ we show how to find the introduced 
representation. Usually, a 
polynomial $p$ is given in the form 
$$p(x)=d_v t^v + d_{v-1} t^{v-1} + \dots + d_{1}t + d_0, \;\; \mbox{ where } d_i \in K.$$
To write $p$ in the form (\ref{darstellungvonp}), 
the occurring coefficients $c_{\nu}$ 
have to be determined. We will show how this can be done for a 
monomial $d_v t^v$. Since $\nu! p_{\nu} = t \cdot (t-1) \cdots (t-\nu+1)$, 
we define 
$$ g^{(\nu)}:=
t\cdot (t-1) \cdots (t-\nu+1) = t^{\nu} + g_{\nu-1}^{(\nu)} t^{\nu-1} + \cdots + g_1^{(\nu)}t.$$
First, the coefficients $g_{v}^{(\nu)}$ will be determined for $1 \leq v \leq \nu$ by 
using the fact that $g^{(\nu)}=g^{(\nu-1)} \cdot (t - \nu +1 )$.
\begin{enumerate} 
\item Determine $g_1^{(\nu)}$:\\
The polynomial $g^{(\nu)}$ is a multiple of $t$.
Recursively, one gets that 
$$g_1^{(\nu)} = 
\left\lbrace  \begin{array}{cc}
1 & \mbox{ for } \nu = 1\\
(-1)^{\nu - 1}  \prod_{k=1}^{\nu-1} k & \mbox{ for } \; \nu > 1. 
\end{array} \right.
$$  
\item Determine $g_2^{(\nu)}$:\\
Using $g^{(\nu)}= g^{(\nu-1)} \cdot (t\nu+1)$, we get
\begin{eqnarray*}
g_2^{(\nu)}& =& g_1^{(\nu-1)} - (\nu-1) \cdot g_2^{(\nu-1)} 
= (-1)^{\nu - 2} \prod_{k=1}^{\nu-2} k     \; - \, (\nu-1) g_2^{(\nu-1)}
\end{eqnarray*}
Since $g_2^{(2)}=1$, we get a recursive formula.
\item Determine $g_j^{(\nu)}$ for $j \leq \nu$:\\
A similar consideration as in the previous point yields
$$ g_j^{(\nu)}=g_{j-1}^{(\nu-1)} - (\nu-1) \cdot g_j^{(\nu-1)}.$$
\end{enumerate}
Finally, we observe 
\begin{eqnarray*}
d_v t^v &=& d_v \left(  g(v) -g_{v-1}^{(v)} \cdot g(v-1) - (g_{v-2}^{(v)}- 
g_{v-1}^{(v)}\cdot g_{v-2}^{(v-1)} )g(v-2) - \cdots \right)     \\
&=& d_v \left( g(v) + \sum_{i=1}^{v-1}k_v(i) \cdot g(v-i) \right)  
= d_v \left( v! p_v + \sum_{i=1}^{v-1}k_v(i) \cdot (v-i)! \cdot p_{v-i} \right) ,
\end{eqnarray*}
where
$$k_v(1):=-g_{v-1}^{(v)}, \;\; \mbox{ and } \;\; 
k_v(l)=\left\lbrace \begin{array}{ll} -g_{v-l}^{(v)}+\sum_{i=1}^{l-1}k_v(i) \cdot g_{v-l}^{(v-i)},  & \mbox{ if } l<v\\
                                       0, & \mbox{ if } l \geq v.
                                        
                 \end{array} \right. $$
Consider for example $p(t)=t^3+t^2+1$. The bounding value $\varrho$ equals three, so by using
\begin{center}
\begin{tabular}{|c|c|c|c|c|c|}
\hline
$j$ & $p_1^{(j)}$ & $p_2^{(j)}$ & $p_3^{(j)}$ & $k_3(j)$ & $k_2(j)$\\
\hline
1 & 1 & 0 & 0 &  3 & 1\\
\hline 
2 & -1 & 1 & 0 & 1 & 0\\
\hline
3 & 2 & -3 & 1 & 0 & 0\\
\hline
\end{tabular}
\end{center}
we finally get
\begin{eqnarray*}
t^3 & = & p(3) + k_3(1) \cdot p(2) + k_3(2) \cdot p(1) = 6 \cdot p_3 + 3 \cdot 2 \cdot p_2 + 1 \cdot p_1 \\
t^2 & = & p(2) + k_2(1) \cdot p(1) = 2 \cdot p_2 + 1 \cdot p_1\\
1   & = & p_0 \\
\Rightarrow \;  p(t) &=& 6 \cdot p_3 + 8 \cdot p_2 + 2 \cdot p_1 + p_0 
\end{eqnarray*}
 
In the following we show the advantage of this notation. Since 
\begin{eqnarray*}
(\delta_i p_{\nu_i})(t_i) & = & \binom{t_i+1}{ \nu_i } - \binom{ t_i }{ \nu_i }\\
& = & \left\lbrace
\begin{array}{cc}
\frac{((t_i+1)-(t_i-\nu_i+1)) \, (t_i \, \cdots \, (t_i-\nu_i+2))}{\nu_i!} & \mbox{ if } \nu_i \geq 1 \\ 
0 & \mbox{ if } \nu_i =0 \end{array} \right. 
\\
& = & \left\lbrace 
\begin{array}{cc} \binom{ t_i}{  \nu_i-1} & \mbox{ if } \nu_i \geq 1 \\ 0 & \mbox{ if } 
\nu_i =0,  \end{array}  \right. 
\end{eqnarray*}
one gets, by using the fact that $\delta^{\mu}p_{\nu}=\delta_1^{\mu_1}p_{\nu_1} \cdots 
\delta_m^{\mu_n}p_{\nu_n}$, the equality 
\begin{eqnarray}\label{shifttozero} \delta^{\mu}p_{\nu} &=& \left\lbrace 
\begin{array}{cc} p_{\nu - \mu} & \mbox{ if } \mu \leq_{cw} \nu  \\ 
0 & \mbox{ otherwise.} \end{array} \right. \label{zero} \end{eqnarray}

\begin{remark}\label{BoundingIndex}
Let $p=(p_1, \dots, p_m)^T  \in A^m$ with $p_i(x)=a_{d_ii}t^{\mu_{di}} + \dots + a_{1i}t^{\mu_{1i}}$, 
using multi-index notation. Define $$\varrho_i=\max_{cw} \left\lbrace (v_1, \dots, v_n) \, \in 
\N^n_0 \; | \; v_j=(\mu_{ki})_j \mbox{ for } 1\leq k \leq d_i \right\rbrace.$$ Then the bounding 
multi-index $\varrho$ belonging to the binomial representation (\ref{darstellungvonp}) is given by
$$ \varrho=\max_{cw}\left\lbrace (v_1, \dots, v_n) \; | \; v_i=(\varrho_j)_i \mbox{ for } 1 \leq j \leq m \right\rbrace.$$
\end{remark}

From now on suppose that $p = \sum_{\nu \in \N^n_0, \, \nu \leq_{cw} \varrho} c_\nu p_\nu$.

\begin{remark}\label{ShiftVmpumfinitelygenerated}
Connecting remark \ref{BoundingIndex} and (\ref{shifttozero}) we get that $\delta^\mu p = 0$ 
for all $\mu$ with $\mu_i > \varrho_i $ for at least one $1 \leq i \leq n$. Now consider the finitely 
generated left $A$-module generated by
$$ \Shift_p={}_{A}\langle \delta^\mu p \; | \; \mu \leq_{cw} \varrho   \rangle .$$
The corresponding syzygy module $\Syz( \Shift_p)$ is finitely generated too, 
since $A$ is a Noetherian ring. Analogously to the continuous case, 
we can give an 
$A$-module homomorphism from $\Syz( \Shift_p)$ to $\ker(\kappa_p)$, 
such that the image of $s_1, \dots, s_d$ under this map  
generates $\ker(\kappa_p)$, that is, $\ker(\kappa_p)$ is 
finitely generated as an $A$-module. 
This implies that $\ker(\kappa_p)$ is finitely generated as an
$\mathcal{S}_n$-module.
\end{remark}

\begin{ex}
\label{exKerP}
 Let $p=[t^3, t]^{T}$. 
Then the continuous $\VMPUM$ is the same as the discrete $\VMPUM$, that is, 
equal to $\{ c [t^3, t]^{T} \mid c\in K\}$. Direct computation over $\mathcal{S}_1$ yields 
$$ \ker_{\mathcal{S}_1}(\kappa_p) = {}_{\mathcal{S}_1}\langle [0, \Delta^2], \; [0, t\Delta-1], \;  [1, -t^2] \rangle$$
and this means that 
$$
\left[ 
\begin{array}{cc}
0 & \Delta^2 \\ 0 & t\Delta-1 \\ 1 & -t^2
\end{array}
\right] 
$$
is a kernel representation of the $\VMPUM$ of $p$. 
Note that over $\mathcal{A}_1$,
we have
$$ \ker_{\mathcal{A}_1}(\kappa_p) = {}_{\mathcal{A}_1}\langle [0, \d^2], \; [0, t\d -1], \;  [1, -t^2] \rangle.$$

Alternatively, we can compute $ \ker_{\mathcal{S}_1}(\kappa_p)$ in the commutative framework, using the analogue of ``difference algebra'' approach. At first, we observe that
$$ 
\Shift_{ [t^3, t]^{T}}= {}_{K[t]} \langle t^3, 3t^2+3t+1, 6t+6, 6, t, 1 \rangle
$$
so 
\begin{footnotesize}
$$
\Syz(\Shift_{ [t^3, t]^{T}})= {}_{K[t]} \langle  
\left[ \begin{array}{c}
0\\0\\0\\0\\1\\-t
\end{array}
\right] ,\;
\left[ \begin{array}{c}
0\\0\\0\\1\\0\\-6
\end{array}
\right] ,\;
\left[ \begin{array}{c}
0\\0\\1\\0\\0\\-6t-6
\end{array}
\right] ,\;
\left[ \begin{array}{c}
0\\1\\0\\0\\0\\-3t^2-3t-1
\end{array}
\right] ,\;
\left[ \begin{array}{c}
1\\0\\0\\0\\0\\-t^3
\end{array}
\right]
\rangle .
$$
\end{footnotesize}

Finally we get that 
$ \ker(\kappa_p) 
 = {}_{\mathcal{S}_1} \langle[0,-t \Delta+1],\;[\Delta^3, -6 \Delta], 
\;[\Delta^2, (-6t-6)\Delta], \;[\Delta,(-3t^2-3t-1)\Delta], \;[1,-t^3\Delta], \;[\Delta^4,0], \;[0,\Delta^2] \rangle 
 = {}_{\mathcal{S}_1}\langle [0, \Delta^2], \; [0, t\Delta-1], \;  [1, -t^2] \rangle.$
\end{ex}

\subsection{VMPUM of polynomial-exponential signals}

For  $ \lambda = ( \lambda_1, \dots, \lambda_n ) \in K^n$ the discrete exponential function is given by
$$
 \exp_{\lambda}: \N_0^n \rightarrow K, \;\; t \mapsto \lambda^t= \lambda_1^{t_1} \cdots \lambda_n^{t_n}.
$$

First suppose that $m=1$, that is, 
we want to construct the $\VMPUM$ of a scalar polynomial exponential 
trajectory of the form $p \exp_\lambda$, where $p \in A$. 
Without loss of generality, we can assume 
$\lambda_i \neq 0 $ for all $1 \leq i \leq n$, since otherwise if $\lambda_j=0$
$$
p \exp_{\lambda} (t) = \left\lbrace 
\begin{array}{cc} 0 & \mbox{ if } t_j \neq 0 \\ g & \mbox{ if } t_j = 0 \end{array}, \;\; \mbox{ where }\right. 
$$
$$
g:\N^{n-1} \rightarrow K, \;\;\;\; t \mapsto  (p \exp_{\lambda}) ( t_1, \dots, t_{i-1}, 0, t_{i+1}, \dots, t_n).
$$ 
 Consider the automorphism of $\mathcal{S}_n$
$$
\chi_{\lambda} : \mathcal{S}_n \rightarrow \mathcal{S}_n, \;\; 
\left\lbrace 
\begin{array}{c} 
 t_i \mapsto t_i \\
 \Delta_i \mapsto \frac{1}{\lambda_i} (\Delta_i - \lambda_i + 1).
\end{array}
\right. \;\;
$$

Since the equality
\begin{align*}
 \chi_{\lambda}(\Delta_i) \bullet (p \exp_\lambda) & = \frac{1}{\lambda_i}(\Delta_i - \lambda_i + 1) \bullet (p \exp_\lambda) \\
 &= \frac{1}{\lambda_i} \left( \lambda_i \exp_\lambda \sigma_i(p) - p \exp_\lambda  - \lambda_i p \exp_\lambda + p \exp_\lambda \right) \\
 & = \frac{1}{\lambda_i} \left( \lambda_i \exp_\lambda ( \sigma_i(p) - p ) \right) 
 = \exp_\lambda \Delta_i \bullet  p
\end{align*}
holds, we obtain the identity
$$  
\chi_{\lambda}(\Delta_i^k) \bullet (p \exp_\lambda) = \exp_\lambda \Delta_i^k \bullet  p
$$
that finally extends to
\begin{align}\label{shiftexp}
\chi_{\lambda}(\Delta^\mu) \bullet (p \exp_\lambda) = \chi_{\lambda}^\mu(\Delta) \bullet (p \exp_\lambda) = \exp_\lambda \Delta^\mu \bullet  p.
\end{align}

Now using (\ref{shiftexp}) we can deduce for $a = \sum_{i=1}^h a_i \Delta^{\alpha_i} \in \mathcal{S}_n$ 
the equivalence
\begin{align}\label{shifteinD}
 a \bullet p = 0 \;\;\; \Leftrightarrow \;\;\; \chi_\lambda(a) \bullet (\exp_\lambda p) = 0,
\end{align}
since
\begin{align*}
 \chi_\lambda(a) \bullet (\exp_\lambda p)  & = \sum_{i=1}^h a_i \chi_\lambda(\Delta^{\alpha_i} )\bullet (\exp_\lambda p)
 = \sum_{i=1}^h a_i (\Delta^{\alpha_i} \bullet p) \exp_\lambda \\
& =\exp_\lambda  \sum_{i=1}^h a_i (\Delta^{\alpha_i} \bullet p) =\exp_\lambda a \bullet p .
\end{align*}

Summarizing, we obtain
\begin{Th}
 Let $R \in \mathcal{S}_n^{l \times 1}$ be a kernel representation matrix of the $\VMPUM$ of $p$. Then the kernel representation matrix of $p \exp_\lambda$ is given by $(\chi_\lambda(R_{i}))_{i}$.
\end{Th}

Now consider 
\begin{align}\label{shiftexpvektor}
p = \left[ \begin{array}{c} p_1 \exp_{\lambda^{(1)}} \\ 
\vdots \\ p_m \exp_{\lambda^{(m)}} \end{array}\right] , \;\; 
\mbox{ where } \;  \lambda^{(i)}  \in K^n \setminus \left\lbrace 0\right\rbrace,  \; \mbox{ and } \; p_i \in A
\end{align}
and suppose $\lambda^{(1)}, \dots, \lambda^{(m)}$ to be pairwise different and without loss of generality 
$\lambda^{(i)}_j \neq 0$ for all $1 \leq i \leq m$ and  $1\leq j\leq n$. Then 
\begin{align}\label{shiftmD}
\sum_{j=1}^m a_j \bullet (p_j \exp_{\lambda^{(j)}}) =0 \;\;\; \mbox{if and only if}\;\;\;[a_1, \dots, a_m] \in 
\bigoplus_{j=1}^m \chi_{\lambda^{(j)}}(\,\ker(\kappa_{p_j})  \,),
\end{align}
which follows by
$$\sum_{j=1}^m a_j \bullet (p_j \exp_{\lambda^{(j)}}) = 0 
\; \Leftrightarrow \; \sum_{j=1}^m \left(  \sum_{i=1}^{h_j} c_{ji} t^{\alpha_{ji}} \Delta^{\beta_{ji}} \bullet (p_j \exp_{\lambda^{(j)}})\right)  = 0 $$
$$\Leftrightarrow   \sum_{j=1}^m (\chi_{\lambda^{(j)}}^{-1}(a_j) \bullet p_j) \exp_{\lambda^{(j)}} =0 
\Leftrightarrow  \chi_{\lambda^{(j)}}^{-1}(a_j) \, \in \, \ker(\kappa_{p_j}) \;\; \mbox{ for all } \, 1 \leq j \leq m.$$

Choosing a suitable projection, we obtain by (\ref{shifteinD}) and (\ref{shiftmD})
\begin{Th}
Let $p$ be of the form (\ref{shiftexpvektor}). Further let 
$K_i:= \left\lbrace j \; | \; \lambda^j= \lambda^i \right\rbrace = 
\left\lbrace  k_{i1}, \dots, k_{i l_i}\right\rbrace $ and $l$ chosen minimal such that the disjoint union 
$K_1 \dot \cup \dots \dot \cup K_l = \left\lbrace k_{11}, \dots, k_{1h_1}, \dots, 
k_{l1}, \dots, k_{lh_l} \right\rbrace =\left\lbrace 1, \dots, m \right\rbrace $. Further define the vector 
$h_i:=[f_{k_{i1}}, \dots, f_{k_{il_i}}]^{T}$ and $H_i:=\chi_{\lambda^{(i)}}( \ker(\kappa_{h_i}) )$. 
Let  $e_{k_{ij}}$ denote the $k_{ij}$-th standard generator of $\mathcal{S}_n^{1 \times m}$ for 
$1\leq i \leq l$ and $1 \leq l \leq h_i$. Defining for $1 \leq i \leq l$
\begin{align*}
\phi_i: H_i \rightarrow \mathcal{S}_n , \;\;\; [a_1, \dots, a_{h_i}] \mapsto \sum_{j=1}^{h_i} a_j e_{k_{ij}} 
\end{align*}
the $\VMPUM$ of $p$ is given by
$ \;\;\bigoplus_{i=1}^l \phi_i(H_i)$. 
\end{Th}

\section*{Conclusion} 
Generalizing ideas from systems theory,
we have defined a ``varying most powerful unfalsified model'' ($\VMPUM$)
over polynomial Ore algebras such as the Weyl algebra or the
difference algebra. Mathematically, this amounts to computing
kernels of module homomorphisms over these algebras. On the
one hand, this can be achieved using Gr\"obner bases techniques, 
and on the other,
by translating the problem to an associated syzygy computation 
over a commutative polynomial ring, thus mimicking ideas of differential
algebra.
We have also studied some structural properties of the resulting models,
and we have seen, in terms of examples, 
that models with polynomial coefficients provide
a much better (and more precise) description of the data than models with constant
coefficients. 
Further future work concerns, for instance, a characterization of the
vector space dimension of the $\VMPUM$ of several trajectories,
thus generalizing Theorem~\ref{CDimension}. 
Let $p=[p_1, \dots, p_m]$ consist of $\C$-linear independent signals.
We conjecture that $\dim_\C(\VMPUM(p))=m$. Moreover, it seems possible to us to develop $\VMPUM$ with polynomial coefficients for data, represented by rational and by rational-exponential functions.

\end{document}